\documentclass[11pt]{article}

\usepackage[utf8]{inputenc}
\usepackage{lmodern}
\usepackage[normalem]{ulem}
\usepackage{float}
\usepackage{geometry}
\usepackage{enumitem}
\usepackage[utf8]{inputenc}
\usepackage{tocloft}
\usepackage{graphicx, titlesec}
\usepackage{amscd,amsmath,amsthm,amssymb,amsfonts}
\usepackage{verbatim}
\usepackage[all]{xy}
\usepackage{color}
\usepackage[colorlinks]{hyperref}
\usepackage{booktabs}
\usepackage{tikz-cd}
\usetikzlibrary{patterns}
\usepackage{ytableau}
\usepackage{booktabs}
\usepackage{cleveref}
\usepackage{natbib}

\usepackage{algorithm}
\usepackage[noend]{algpseudocode}
\usepackage{algorithmicx}
\usepackage{algorithm}
\usepackage{algpseudocode}

\MakeRobust{\Call}

\algdef{SE}[DOWHILE]{Do}{DoWhile}{\algorithmicdo}[1]{\algorithmicwhile\ #1}
\algrenewcommand{\algorithmiccomment}[1]{\hfill $\rhd$ \emph{#1}}
\algrenewcommand{\algorithmicrequire}{\textbf{Input:}}
\algrenewcommand{\algorithmicensure}{\textbf{Output:}}
\algnewcommand{\Or}{\textbf{or}}
\algnewcommand{\And}{\textbf{and}}
\algnewcommand{\Not}{\textbf{not}\,}
\algnewcommand\algorithmicforeach{\textbf{for each}}
\algdef{S}[FOR]{ForEach}[1]{\algorithmicforeach\ #1\ \algorithmicdo}

\DeclareMathOperator{\Gr}{Gr}

\DeclareMathOperator{\rank}{rk}

\DeclareMathOperator{\sign}{sign}

\DeclareMathOperator{\rowspan}{rowspan}
\DeclareMathOperator{\Id}{Id}

\DeclareMathOperator{\alt}{alt}
\DeclareMathOperator{\OGr}{OGr}
\DeclareMathOperator{\SGr}{SGr}

\DeclareMathOperator{\diag}{diag}
\DeclareMathOperator{\sd}{sd}
\DeclareMathOperator{\GL}{GL}
\DeclareMathOperator{\PGL}{PGL}

\DeclareMathOperator{\SO}{SO}

\newcommand{\mM}{\mathcal{M}}

\newcommand{\mB}{\mathcal{B}}

\newcommand{\mS}{\mathcal{S}}
\newcommand{\mR}{\mathcal{R}}
\newcommand{\mU}{\mathcal{U}}

\newcommand{\mV}{\mathcal{V}}
\newcommand{\rk}{\textrm{rk }}

\newcommand{\RR}{\mathbb{R}}
\newcommand{\CC}{\mathbb{C}}
\newcommand{\PP}{\mathbb{P}}
\newcommand{\KK}{\mathbb{K}}
\newcommand{\bu}{\mathbf{u}}

\newcommand{\sslash}{\mathbin{/\mkern-6mu/}}
\newcommand{\lsslash}{\mathbin{\setminus\mkern-6mu\setminus}}

\newtheorem{theorem}{Theorem}[section]
\theoremstyle{definition}
\newtheorem{definition}[theorem]{Definition}

\newtheorem{example}[theorem]{Example}
\newtheorem{remark}[theorem]{Remark}
\newtheorem{question}[theorem]{Question}
\theoremstyle{plain}
\newtheorem{lemma}[theorem]{Lemma}
\newtheorem{proposition}[theorem]{Proposition}
\newtheorem{corollary}[theorem]{Corollary}

\newtheorem{innercustomthm}{Theorem}
\newenvironment{customthm}[1]
{\renewcommand\theinnercustomthm{#1}\innercustomthm}
{\endinnercustomthm}

\usepackage{listings}
\lstdefinelanguage{Julia}%
  {morekeywords={abstract,break,case,begin,catch,const,continue,do,else,elseif,%
      end,export,false,for,function,immutable,import,importall,if,
      macro,module,otherwise,quote,return,switch,true,try,type,typealias,%
      using,while},%
   sensitive=true,%
   alsoother={\$},%
   morecomment=[l]\#,%
   morecomment=[n]{\#=}{=\#},%
   morestring=[s]{"}{"},%
   morestring=[m]{'}{'},%
}[keywords,comments,strings]
\title{The Self-Projecting Grassmannian}
\author{Alheydis Geiger\thanks{Max Planck Institute for Mathematics in the Sciences, Leipzig \hfill {\tt geiger@mis.mpg.de}} \, and Francesca Zaffalon\footnotemark[1]\;\,\thanks{Weizmann Institute of Science \hfill {\tt zaffalon@mis.mpg.de}}}
\date{}

\begin{document}

\maketitle

\begin{abstract}
We introduce the self-projecting Grassmannian, an irreducible subvariety of the Grassmannian parametrizing linear subspaces that satisfy a generalized self-duality condition. We study its relation to classical moduli spaces, such as the moduli spaces of pointed curves of genus $g$, as well as to other natural subvarieties of the Grassmannian. We further translate the self-projectivity condition in the combinatorial language of matroids, introducing self-projecting matroids, and we computationally investigate their realization spaces inside the self-projecting Grassmannian.
\end{abstract}

\section{Introduction}

The Grassmannian is the point of entry to moduli spaces. It parameterizes $k$-dimensional vector subspaces of $\KK^n$, or configurations of $n$ points in $\PP^{k-1}$, modulo $\PGL$-action. 
Subvarieties of the Grassmannian encode additional constraints on such point configurations. One important example is self-duality, where the set of points satisfies Gale duality. 
These were studied, among others, by Dolgachev--Ortland, Coble, Petrakiev and Eisenbud--Popescu \cite{dolgachev1988point, coble, petrakiev, eisenbud2000projective}. The subvariety of such point configurations inside the Grassmannian is the self-dual Grassmannian~\cite{GHSV24}. 
Formally, a point $V\in \Gr(k,2k)$ is said to be self-dual if $V$ is an element of the $(\KK^*)^{2k}$ torus orbit of $V^{\perp}$.

In this paper, we generalize self-dual point configurations to {\em self-projecting point configurations}. These describe points $V\in \Gr(k,n)$ such that $V$ is {\em contained} in an element of the $(\KK^*)^n$ torus orbit of $V^{\perp}$. 
We introduce the subvariety of the Grassmannian parameterizing such point configurations: the {\em self-projecting Grassmannian}. It is defined as the Zariski closure of the space of $V\in \Gr(k,n)$, represented by a $k\times n$ matrix $X$, such that
\begin{equation}
    \exists \text{ diagonal full-rank matrix } \Lambda \text{ such that } X\cdot \Lambda \cdot X^t = 0. \label{selfproj}
\end{equation}
\begin{customthm}{\ref{thm : irreducible}}
   If $2k\leq n \leq \binom{k+1}{2}$, the~self-projecting~Grassmannian~$\SGr(k,n)\subseteq\Gr(k,n)$ is irreducible of dimension $k(n-k)-\binom{k+1}{2}+n-1$.
    For $n < 2k,$ we have $\SGr(k,n) = \emptyset$ and for $n > \binom{k+1}{2}$ we have $\SGr(k,n)=\Gr(k,n)$.
\end{customthm}
In particular, each self-projecting point configuration in $\Gr(k,n)$ lies in some orthogonal Grassmannian $\OGr^{\Lambda}(k,n)$, defined by a diagonal bilinear form $\Lambda$. 
Therefore, the self-projecting Grassmannian encodes in a single object all vector subspaces which are isotropic with respect to some non-degenerate diagonal bilinear form; see~\Cref{lemma : irreducible}.
Orthogonal Grassmannians, and their totally positive parts, play an important role in applications, where they appear, among others, in the study of scattering amplitudes in ABJM theory~\cite{huang2014abjm,huang2014positive}, in Ising models~\cite{Galashin_2020} and in cosmology~\cite{MR2607588}.

The self-projectivity condition on a space $V\in \Gr(k,n)$, represented by a $k\times n$ matrix~$X$, can be written as 
\[ 
    \exists \; \lambda\in (\KK^*)^n \text{ such that } \nu(X)\cdot \lambda = 0, 
\]
where $\nu(X)$ is obtained by applying the second Veronese embedding $\nu$ to each column of~$X$. 
This definition highlights the relation to the characterization of self-duality via intersection of quadrics, and generalizes it in a natural way. 
Moreover, it reveals a natural connection to the parameter space $X_{r,d^m,n}$ of point configurations on intersections of hypersurfaces, as studied in~\cite{caminata2023determinantal}. 

Whenever a family of genus $g$ curves can be generically written as the intersection of quadric surfaces in the correct embedding, the corresponding moduli space $\mM_{g,n}$ will be birational to a self-projecting Grassmannian.

\begin{customthm}{\ref{thm : X49} and \ref{thm : X513}}
    The moduli space $\mM_{1,10}$ of genus $1$ curves with $10$ marking and the self-projecting configuration space $\SGr(4,9)^\circ/(\CC^*)^9$, as well as $\mM_{5,13}$ and $\SGr(5,13)^\circ/(\CC^*)^{13}$, are birationally equivalent.
\end{customthm} 

Self-duality, in the classical $n=2k$ case, has a clear combinatorial translation to the concept of (identically) self-dual matroids. We introduce the class of {\em self-projecting matroids} encoding the combinatorial constraints that need to hold in a self-projecting vector space. For small values of $k,n$ we compute the number of such matroids, in~\Cref{tab:matroids}. 

Realization spaces of self-dual matroids inside the self-dual Grassmannian were studied in~\cite{GHSV24}. They show that there exist realizable self-dual matroids that do not admit a self-dual realization, where the first example is a rank $5$ matroid on $10$ elements. Generalizing this study, we compute self-projecting realization spaces $\mS(\mM)$ of self-projecting matroids $\mM$ and compare with their realization spaces $\mathcal{R(\mM)}$.

\begin{customthm}{\ref{thm : rank2}, \ref{thm : realization space rank 3}~and~\ref{thm : realization space rank 4}}
    For all self-projecting matroids $\mM$ of rank~$2$, we have $\mS(\mM) = \mR(\mM)$.

    The variety $\mS(\mU_{3,6})$ is a codimension $1$ subvariety of $\mR(\mU_{3,6})$. For all but $4$ other self-projecting matroids of rank $3$ on up to $8$ elements, $\mS(\mM)=\mR(\mM)$.

    Out of the $7181$ realizable self-projecting matroids of rank $4$ on $9$ elements there are at least $174$ for which $\mS(\mM) = \mR(\mM)\neq \emptyset$, while for at least $5400$ we have $\mS(\mM)\subsetneq\mR(\mM)$. For at least $2844$ realizable, self-projecting matroids, $\emptyset \neq \mS(\mM) \subsetneq \mR(\mM)$. 
\end{customthm}

Our computations are made available according to the FAIR data principles broadcasted by \href{https://www.mardi4nfdi.de/about/mission}{MaRDI}~\cite{Mardi}, which aim to make research data Findable, Accessible, Interoperable, and  Reusable.
To guarantee transparent and verifiable results, the computed realization spaces are stored 
using the new \texttt{.mrdi} file format \cite{fairfileformat}, designed to store and share results in computer algebra without losing accuracy. They can be found at
\begin{center}\url{https://github.com/AlheydisGeiger/selfprojectingGrassmannian}
\end{center} 
The code written in \texttt{OSCAR} will be incorporated into the experimental section of the general \texttt{OSCAR}~\cite{OSCAR,OSCAR-book} distribution in the near future, together with a collection for the newly emerging database \texttt{oscarDB}. See \Cref{sec : computations} for further details on the computations.

The database collection also includes the results from \cite{GHSV24} in a more accessible way, as well as producing the infrastructure to store and load (self-projecting) realization spaces of matroids. This framework makes it possible, for example, to test these spaces for smoothness. This is a topic of interest as Mn\"ev's universality theorem shows that the algebraic structure of matroid realization spaces can be as bad as possible~\cite{mnev}. However,~\cite{CL25} shows that small matroid realization spaces are smooth. 

\medskip

\noindent\textbf{Outline.} The paper is structured as follows. After recalling the definitions, we define the self-projecting Grassmannian in~\Cref{sec : definitions} and study its relation to other well-studied varieties: the orthogonal Grassmannian $\OGr(k,n)$ in~\Cref{sec : orthogonal}, its quotient modulo torus action $X(k,n)^{\sd}$ in~\Cref{sec : point configuration} and its relation to parameter spaces of points lying on the intersection of certain hypersurfaces, as introduced in \cite{caminata2023determinantal}, in~\Cref{sec : point conf hypersurfaces}. 
In~\Cref{sec : geom int} we give a geometric interpretation of some small self-projecting Grassmannians. In particular we study the smallest new example of self-projecting Grassmannian $\SGr(4,9)$ and its relation to the moduli space of marked elliptic curves in~\Cref{sec : X49}. In~\Cref{sec : X410} we remark that $\SGr(4,10)$ is a codimension $1$ subvariety parametrizing point configurations at the center of the Bruxelles' problem. In~\Cref{sec : X513} we study the relation between $\SGr(5,13)$ and the moduli space of genus $5$ curves, with $13$ marked points. 
In~\Cref{sec : 4 matroids}, we turn our attention to matroids that can arise from self-projecting points and introduce the definition of self-projecting matroids in~\Cref{sec : sp matroids}. We computationally study the realization spaces of self-projecting matroids inside the self-projecting Grassmannian in~\Cref{sec : sp realizaitons}. 
In~\Cref{sec : reality and positivity} we consider the real and positive part of the self-projecting Grassmannian. We conclude in~\Cref{sec : sp positroids} with some remarks on self-projecting positroids showing they are not always realizable in the totally non-negative self-projecting Grassmannian. 

\noindent\textbf{Acknowledgments.} 
The authors thank Ben Hollering, Michael Joswig, Nathan Pflueger, Luca Schaffler, Rainer Sinn, and Bernd Sturmfels for insightful discussions. Thanks for help with setting up the \texttt{OSCAR} code and database collection go to Antony Della Vecchia. The first author was funded by the Deutsche Forschungsgemeinschaft (DFG, German Research Foundation) – project number 539846931, within the SPP 2458~\emph{Combinatorial Synergies}.

\section{The self-projecting Grassmannian}\label{sec : 2}

In this section we introduce the central object of study, the self-projecting Grassmannian, and investigate its relation to other parameter spaces of point configurations.

The Grassmannian $\Gr(k,n)$ is the set of $k$-dimensional subspaces of an $n$-dimensional vector space over a field $\KK$. 
It can be embedded into projective space via the {\em dual Pl\"ucker embedding}, mapping a $k$-dimensional vector space $V$ to the vector of maximal minors $(q_I(V))_{I\in \binom{[n]}{k}}$ of a $k\times n$ matrix $M$ with $V=\rowspan(M)$. By abuse of notation, we denote the matrix $M$ also by $V$.

\subsection{Definitions}\label{sec : definitions}

The self-dual Grassmannian, as introduced in~\cite{GHSV24}, is the subvariety of the Grassmannian $\Gr(k,2k)$ defined as the Zariski closure of the set $\SGr^\bullet(k,2k)$  of all vector spaces $V\in \Gr(k,2k)$ such that 
\begin{equation} 
    q_{[n]\setminus I}(V) = \lambda_I \sign(I) q_I(V), \quad \text{ for all } I \in \tbinom{[2k]}{k},\label{eq : selfdual} 
\end{equation}
where $\lambda_I = \prod_{i\in I}\lambda_i$.
An element $V \in \SGr^\bullet(k,2k)$ is said to be \emph{self-dual}. Note that for such an element $V$ and the corresponding choice of $\lambda \in (\KK^*)^{2k}$, we have that $V^\perp = \lambda \cdot V$ or, equivalently, $V\cdot \diag(\lambda)\cdot V^t = 0$.

\begin{definition}\label{def : self-proj}
    A subspace $V\in \Gr(k,n)$ is {\em self-projecting} if there exists $\lambda \in (\KK^*)^n$ such that $V\cdot \diag(\lambda)\cdot V^t = 0$.
    Denote by $\SGr^\bullet(k,n)$ the set of self-projecting points in $\Gr(k,n)$.
    The {\em self-projecting Grassmannian} $\SGr(k,n)$ is the Zariski closure of $\SGr^\bullet(k,n)$ inside the Grassmannian $\Gr(k,n)$.
\end{definition}
Since the set of self-projecting points is not empty if and only if $n\geq 2k$, from now on we assume $n\geq 2k$. Geometrically, a subspace $V$ is self-projecting if it is contained in some element of the orbit of the $(\KK^*)^n$ torus action on its orthogonal complement $V^\perp$. 

Equivalently, a space is self-projecting, if there exists some $\lambda\in (\KK^*)^n$ satisfying a system of linear equations as follows.
Let $\nu: \KK^\ell \to \KK^{\binom{\ell+1}{2}}$ be the second Veronese embedding.
Given an $\ell\times m$ matrix $M$, we denote by 
$\nu(M) = \begin{pmatrix} \nu(v_1) \cdots \nu(v_m) \end{pmatrix}$ 
the $\binom{\ell+1}{2}\times m$ matrix obtained by applying the Veronese embedding to each column of $M$. We refer to this as the \emph{multi-Veronese matrix} of $M$. Given an element $V\in \Gr(k,n)$ represented by a $k\times n$ matrix $M_V$, we can consider the $\binom{k+1}{2}\times n$ matrix $\nu(M_V)$. Then $V$ is self-projecting if and only if there exists a solution $\lambda\in (\KK^*)^n$ to
\[ 
    \nu(M_V)\cdot \lambda = 0. 
\]
The above construction is independent of the choice of representative of the point $V$. Indeed, suppose $M_W$ is another $k\times n$ matrix representing $V$. Then there exists some $G\in \GL(k)$ such that $M_W = G\cdot M_V$. Consider the second symmetric power $G^{[2]}$ of $G$. 
By construction, $\nu(M_W) = G^{[2]}\cdot \nu(M_V)$. The  matrix $G^{[2]}$ is invertible, and its inverse is $(G^{-1})^{[2]}$.
Hence, the system of equations defined by $\nu(M_V)$ does not depend on the choice of the representative for $V$. Therefore, we will write $\nu(V)$ for $\nu(M_V)$.

\begin{proposition}\label{prop : ideal}
    The self-projecting Grassmannian can be defined in dual Stiefel coordinates as the variety $\mathcal{V}(I^{\sd}_{k,n})$ with
    \begin{equation}
        I^{\sd}_{k,n} = \left(\langle \nu(X)\cdot \lambda \rangle: \langle \lambda_1\cdot\ldots\cdot \lambda_n \rangle^\infty\right) \cap \KK[x_{(1,1)}, \ldots,x_{(k,n)}], \label{eq: defining ideal}
    \end{equation}
    where $X$ denotes the matrix with entries $X_{ij} = x_{(i,j)}$. 
\end{proposition}

In order to introduce the ideal of the self-projecting Grassmannian in dual Pl\"ucker coordinates, we define an auxiliary matrix.
\begin{definition}\label{def : cocircuit matrix}
    The {\em cocircuit matrix} $D_{k,n}$ is the $\binom{n}{k-1}\times n$ matrix with 
    \[ 
        (D_{k,n})_{I,j} = \begin{cases} \sign(I,j)q_{I\cup j} \quad & \text{ if } j\not\in I\\
        0 & \text{ else}
        \end{cases} \qquad \text{for } I\in \tbinom{[n]}{k-1},j\in [n].
    \]
\end{definition}

It was shown in~\cite[Remark 2.3]{maazouz2024positive} that given a diagonal bilinear form  represented by a matrix $\Lambda$, a point $V\in \Gr(k,n)$ is isotropic with respect to that bilinear form if and only if $D_{k,n}(V)\cdot \Lambda \cdot D_{k,n}(V)^t = 0$. The matrix $D_{k,n}$ has rank $k$,~\cite[Remark 2.2]{devriendt2025two}.
Hence, a point $V\in \Gr(k,n)$ is self-projecting if and only if there is a solution in the algebraic torus $(\KK^*)^n$ to the system of linear equations
\[ 
    \nu(D_{k,n}(V)) \cdot \lambda = 0. 
\] 

\begin{corollary}
    In dual Plücker coordinates, the self-projecting Grassmannian $\SGr(k,n)$ is defined as the variety of the ideal
    \begin{align*}
        J^{\sd}_{k,n}:= \left(\left\langle \nu(D_{k,n}) \cdot \lambda\right\rangle : \langle \lambda_1\cdots\lambda_n \rangle ^\infty\right)\cap \KK\left[q_I\mid I \in \tbinom{[n]}{k}\right] +I_{k,n},
    \end{align*} 
    where $I_{k,n}$ is the ideal of Plücker relations defining the Grassmannian $\Gr(k,n)$.
\end{corollary}

For any self-projecting point $V\in \SGr(k,n)$ the matrix $\nu(V)$ has rank at most $n-1$. As shown  in examples in~\cite{caminata2023determinantal}, the ideal generated by the $n$-minors of $\nu(X)$ inside $\KK[x_{i,j}]$ might have many irreducible components. However, by~\Cref{thm : irreducible} the self-projecting Grassmannian is an irreducible variety.

\begin{proposition}\label{prop : relation with minors}
    Let $2k\leq n \leq \binom{k+1}{2}$. The self-projecting Grassmannian $\SGr(k,n)$ is equal to the top-dimensional irreducible component of the vanishing locus of the ideal $I_n$ of $n$-minors of $\nu(X)$ in dual Stiefel coordinates.
    
    In dual Pl\"ucker coordinates it is equal to the top-dimensional irreducible component of the vanishing locus of the ideal $J_n$ of $n$-minors of $\nu(D_{k,n}|_{[k]})$, where $D_{k,n}|_{[k]}$ denotes the submatrix of $D_{k,n}$ given by the first $k$ rows.
\end{proposition}
\begin{proof}
    Consider the variety $\mV(I_{n,n-1})$, where $I_{n,n-1}$ is the ideal of $n$-minors of $\nu(X)$, saturated by the ideals of $(n-1)$-minors of $\nu(X)|_S$, for all $S\in \binom{[n]}{n-1}$. A point $V$ is such that $\nu(V)$ has an element of the right kernel in the torus if and only if $\nu(V)$ is rank deficient and all submatrices obtained by deleting a column have full rank equal to $n-1$. Therefore $\mV(I_{n,n-1}) = \SGr(k,n)$. Since $\SGr(k,n)$ is irreducible by~\Cref{thm : irreducible}, and the vanishing locus of the ideal of $(n-1)$-minors of $(\nu(X)|_S)_S$ has dimension strictly smaller than $\dim(\SGr(k,n))$, the first claim follows.
    
    As above, $\SGr(k,n)$ is equal to the variety of $V\in \Gr(k,n)$ in Pl\"ucker coordinates such that $\nu(D_{k,n}(V))$ is rank deficient but all $(n-1)$ submatrices obtained by deleting one column have full rank equal to $n-1$. If all $(n-1)$-minors of $\nu(D_{k,n})$ are non-vanishing, $q_{[k]}\neq 0$. Then the submatrix of $D_{k,n}$ given by the first $k$ rows, i.e. by the rows labeled by subsets in $\binom{[k]}{k-1}$, is full-rank. Hence, the second statement follows.
\end{proof}

\subsection{The orthogonal Grassmannian}\label{sec : orthogonal}

The self-projecting Grassmannian is closely related to another class of highly interesting subvarieties of the Grassmannian, the orthogonal Grassmannians~\cite{maazouz2024positive}. 
Given a nondegenerate symmetric bilinear form $\omega$, the corresponding orthogonal Grassmannian encodes the spaces which are isotropic with respect to $\omega$. The self-projecting Grassmannian is the Zariski closure of the union over all orthogonal Grassmannians defined by a {\em diagonal} bilinear form. 

\begin{definition}
    Let $\omega$ be a nondegenerate symmetric bilinear form, then the orthogonal Grassmannian $ \OGr^\omega(k,n)$ is the variety of $V\in \Gr(k,n)$ with $\omega(u,v)=0 \text{ for all } u,v\in V$.
\end{definition}

In the following we consider orthogonal Grassmannians with respect to diagonal bilinear forms. Given $\lambda\in (\KK^*)^{n}$, denote by $\eta_\lambda$ the bilinear form 
\[ \eta_\lambda(x,y) = \lambda_1x_1y_1 + \lambda_2x_2y_2+ \cdots + \lambda_{n}x_{n}y_{n}, \]
corresponding to the matrix $\diag(\lambda)$. We write $\OGr^\lambda(k,n)$ for $\OGr^{\eta_{\lambda}}(k,n)$. By definition, we have set-theoretic equality
\begin{equation}
    \SGr^\bullet_\KK(k,n) = \bigcup_{\lambda\in (\KK^*)^{n}} \OGr_\KK^\lambda(k,n). \label{eq: SGr and OGr}
\end{equation} 

For the study of the orthogonal Grassmannian over $\CC$ the choice of the bilinear form does not make a difference. However, over $\RR$ different choices of bilinear forms can behave very differently; see~\Cref{sec : reality and positivity} for more details. In the literature, the bilinear form is usually fixed to be either the diagonal form given by $\alt(1) = ((-1)^{i-1})_{i=1,\ldots,n}\in (\KK^*)^n$ or the form $\omega$ with matrix 
$\begin{pmatrix} 0 & \Id_k\\ \Id_k & 0 \end{pmatrix}$.

We generalize \cite[Proposition 5.1]{Galashin_2020} for the $n=2k$ case for arbitrary nondegenerate diagonal bilinear forms.
\begin{proposition}\label{prop : orthogonal generalized}
    For $V\in \Gr(k,2k)$ the following are equivalent.
    \begin{enumerate}[label=(\roman*)]
        \item $V\in \OGr^\lambda(k,2k)$;
        \item $q_{[2k]\setminus I}(V) = \pm\sqrt{\frac{1}{\lambda_{[2k]}}} \sign(I, [2k]\setminus I)\lambda_{I}q_{I}(V)$, where $\pm$ is the same for all $I \in \binom{[2k]}k$.
    \end{enumerate}
\end{proposition}
\begin{proof} Let $\lambda\in(\KK^*)^n$.
    We have $V\in \OGr^\lambda(k,2k)$ if and only if $\langle u, \lambda\cdot v\rangle = 0$ for all $u,v\in V$, or equivalently $V^\perp = \lambda\cdot V$. Hence, there exists $c\in \KK^*$ with
    \begin{align*}
        q_I(V) = c \cdot q_{[2k]\setminus I}(\lambda\cdot V)= c \cdot \sign([2k]\setminus I, I)\lambda_{[2k]\setminus I}q_{[2k]\setminus I}(V),
    \end{align*}
    where $\lambda_{I}=\prod_{i\in  I} \lambda_i$. Applying the equality twice, it follows that $c^2\lambda_{[2k]}=1$, hence the statement follows.
\end{proof}

In particular, it follows that $\OGr^\lambda(k,2k)$ has two connected components: $\OGr^\lambda_{+}(k,2k)$ of points $V\in \OGr^\lambda(k,2k)$ such that the equality {\it (ii)} holds with a plus and $\OGr^\lambda_{-}(k,2k)$ where {\it (ii)} holds with a minus.

\begin{remark}\label{rem : components}
   One can ask if it is possible to always find a $\lambda'$ such that $\OGr^\lambda_-(k,2k)$ is equal to $\OGr_+^{\lambda'}(k,2k)$. 
    If $k$ is odd, one can take $\lambda' = - \lambda$. This works for any base field.
    If $k$ is even and the base field is algebraically closed, we can choose a $\lambda' = i \lambda$, where $i$ is a square root of $-1$. 
    Over $\RR$ there is no way to achieve the equality $\OGr^\lambda_-(k,2k) = \OGr_+^{\lambda'}(k,2k)$ if $k$ is even.
\end{remark}

\begin{lemma}\label{lemma : irreducible} If $\KK$ is algebraically closed, then $\SGr_\KK(k,n)$ is the Zariski closure of the image of $\OGr_\KK^{\mathbf{1}}(k,n)$ under the $(\KK^*)^n$ torus action. 

If $\KK = \RR$, then $\SGr_\RR(k,n)$ is the Zariski closure of the image of $\bigcup_{\bu \in \{\pm 1\}^{n}}\OGr_\RR^{\bu}(k,n)$ under the $(\RR_{>0})^n$ action.
\end{lemma}
\begin{proof}
    Let $\KK$ be algebraically closed. We show that the image of the torus action on $\OGr^{\mathbf{1}}_\KK(k,n)$ is equal to the set-theoretic union of all orthogonal Grassmannians with respect to some diagonal bilinear form. Let $V\in \OGr^{\mathbf{1}}_\KK(k,n)$, and $\lambda\in (\KK^*)^n$. Choose $\mu\in (\KK^*)^n$ with $\mu_i =\frac{1}{\lambda_i^2}$. Then, the space $\lambda\cdot V$ lies in $\OGr^{\mu}_\KK(k,n)$.
    Conversely, given $\lambda\in (\KK^*)^n$ and $V\in \OGr^{\lambda}(k,n)$, choose $\mu\in (\KK^*)^n$ such that $\mu_i^2 =\lambda_i$. 
    Then, $\mu\cdot V$ lies in $\OGr^{\mathbf{1}}_\KK(k,n)$.
    
    Let $\KK= \RR$. Let $V\in\SGr^\bullet(k,n)$  and $\lambda\in(\RR^*)^n$ such that $V\diag(\lambda)V^t =0$. There exists a unique $\bu \in \{\pm 1\}^{n}$ and $\sqrt{|\lambda|}\in (\RR_{>0})^n$, with  $\lambda_i = \bu_i \cdot (\sqrt{|\lambda|})_i^2$. Then $\frac{1}{\sqrt{|\lambda|}}\cdot V \in \OGr^\bu(k,n)$.
    Conversely, let $V\in \OGr^\bu(k,n)$ for $\bu\in \{\pm 1\}^n$ and let $\lambda\in(\RR_{>0})^n$. Then $\lambda V$ is self-projecting with respect to $\mu_i = u_i \cdot\frac{1}{\lambda_i^2}$. 
\end{proof}

\begin{theorem}\label{thm : irreducible}
    The variety $\SGr(k,n)$ 
    is absolutely irreducible.
\end{theorem}
\begin{proof}
   Taking the image under the torus action as in \Cref{lemma : irreducible} preserves irreducibility, so the statement follows for $n>2k$ from the irreducibility of the orthogonal Grassmannian. For $n = 2k,$ we consider that in \Cref{lemma : irreducible} we can use $\OGr_+(k,2k)$ instead of $\OGr(k,2k)$, since the torus action maps $\OGr_+(k,2k)$ to $\OGr_-(k,2k)$ as described in \Cref{rem : components}. As $\OGr_+(k,2k)$ is irreducible, the claim follows.
\end{proof}
We can compute the dimension of the self-projecting Grassmannian using the known dimensions of the orthogonal Grassmannian.

\begin{corollary}\label{cor : dim}
    Let~$\KK$~be~algebraically~closed~or~$\RR$. For~$n > \binom{k+1}{2}$,~$\SGr(k,n)=\Gr(k,n)$. For $2k\leq n \leq \binom{k+1}{2}$,
    \[\dim (\SGr_{\KK}(k,n)) =  k(n-k)-\tbinom{k+1}{2}+n-1.\]
\end{corollary}
\begin{proof}
    The statement follows from \cite[Proposition 2.4]{maazouz2024positive} and~\Cref{lemma : irreducible}. If the base field is $\RR$, these results can be applied, since the quadratic form used in the proof of \cite[Proposition 2.4]{maazouz2024positive} is diagonalizable over $\RR$. In particular, $\OGr^{\alt(1)}(k,n)$ has real dimension equal to its complex dimension.
\end{proof}

\subsection{Point configurations $X(k,n)$}\label{sec : point configuration}

Geometrically, self-projecting vector spaces $V\in \Gr(k,n)^\circ$ are configurations of $n$ points in $\PP^{k-1}$ such that their dual configuration in $\PP^{n-k-1}$ can be projected to the original one. 

\begin{definition}
    The {\em self-projecting configuration space} $X_\KK(k,n)^{\sd}$ is the very affine variety $\SGr_\KK(k,n)^\circ/(\KK^*)^n$ inside $X_\KK(k,n) = \Gr_\KK(k,n)^\circ/(\KK^*)^n$.
\end{definition}

The self-projecting configuration space and the orthogonal Grassmannian are related as follows.

\begin{proposition}\label{thm : XSD and OGr}
    For $\KK$ algebraically closed, there is a dominant finite map between open subsets of the very affine varieties~$\OGr^{\mathbf{1}}(k,n)\,\text{and}\,X_\KK(k,n)^{\sd}$.
\end{proposition}
\begin{proof}
    Consider the open subset $U_S$ of points $V\in \SGr^\bullet_\KK(k,n)$ such that $\ker \nu(V) = \langle \lambda \rangle$ for some $\lambda\in (\KK^*)^n$. Note that this set is closed under torus action and we can consider the quotient $U_X = U_S/(\KK^*)^n$ as an open subset of $X_\KK(k,n)^{\sd}$. Similarly, consider the open subset $U_O$ of points $V \in \OGr^\mathbf{1}_\KK(k,n)^\circ$ such that $\ker(\nu(V)) = \langle \mathbf{1}\rangle$. Consider the map $\varphi:U_O \to U_X$, sending a space $V\in U_O$ to its equivalence class up to torus action inside $U_X$. That the map is surjective and each fiber contains $2^{n-1}$ points follows similarly to~\Cref{lemma : irreducible}.
\end{proof}

\begin{remark}
    The birational parametrization of $X(k,2k)^{\sd}$ via the special orthogonal group $\SO(k)$ as stated in~\cite[Theorem 2.4]{GHSV24} is not correct. Indeed, $\SO(k)$ is birational to one of the connected components of the orthogonal Grassmannian $\OGr(k,2k)$. Therefore, the map from $\SO(k)$ to $X(k,2k)^{\sd}$ defined by sending an orthogonal matrix $R$ to $(I\mid R)$ is only finite to one, but not injective. 
    However, the main consequence of the Theoreom,~\cite[Corollary 2.5]{GHSV24}, was proven in~\cite[Theorem 4, p.51]{dolgachev1988point} and remains correct.
\end{remark}

\begin{proposition}\label{thm : real XSD and OGr}
    There is a dominant finite map between open subsets of the very affine varieties $\bigcup_{\bu\in\{1\}\times \{\pm 1\}^{n-1}} \OGr^\bu _\RR(k,n)^\circ$ and $X_\RR(k,n)^{\sd}$. 
\end{proposition}
\begin{proof}
    Consider the same open subsets $U_S,U_X$ as in the proof of~\Cref{thm : XSD and OGr}. Let $U_O$ be the union of the open subsets of points $V \in \OGr^\bu_\RR(k,n)^\circ$ such that $\ker(\nu(V)) = \langle \bu \rangle$, for all $\bu\in \{1\}\times\{\pm 1\}^{n-1}$. As before, we can define a map $\varphi: U_O \to U_X$ by mapping a point $V\in U_O$ to its equivalence class up to the action of $(\RR_{>0})^n$. That $\varphi$ is surjective and each fiber contains $2^{n-1}$ points follows similarly to~\Cref{lemma : irreducible}.
\end{proof}

\subsection{Associated point configurations}

The study of (self-)dual point configurations has a rich history. Introduced in 1922 by Coble \cite{coble}, they were studied among others by Dolgachev--Ortland, Petrakiev, Kapranov~\cite{dolgachev1988point,petrakiev,kapranov93chow}.

Recall that an \emph{associated} point configuration is a pair $(x,y)\in(\PP^{k-1})^n\times (\PP^{n-k-1})^n$, such that $(\phi(x_i,y_i))_{i \in[n]}\in(\PP^{k(n-k)-1})^n$ is a minimal projectively dependent set, where $\phi:\PP^{k-1}\times \PP^{n-k-1}\to \PP^{k(n-k)-1}$ is the Segre embedding.
 
\begin{proposition}
    A self-projecting point configuration $X\in X(k,n)^{\sd}$ defines associated points $(x,y)\in (\PP^{k-1})^n\times (\PP^{n-k-1})^n$, with $\pi(y_i) = x_i$, for all $i \in [n]$, such that the equivalence class of $x$ in $X(k,n)^{\sd}$ is equal to $X$.
    Here $\pi: \PP^{n-k-1}\to \PP^{k-1}$ is the projection to the first $k$ coordinates.
\end{proposition}
\begin{proof}
    The claim follows from the following observation. Let $X,Y$ be matrices with representatives of the projective points $x_i, y_i$ in the columns, respectively. Then $x,y$ are associated if and only if there exists a diagonal matrix $\Lambda$ with non-zero entries such that $X\Lambda Y^t =0$. Equivalently, $Y= \Lambda\cdot \ker(X)$.
    For $X\in \SGr^\bullet(k,n)$ with $X\Lambda X^t =0$, we have $X \subseteq \Lambda\ker(X)$. Take $Y=\ker(X)$ represented by a matrix containing in the first $k$ rows a basis of $X$.
\end{proof}

The equivalent definition of associated point configurations via the existence of a full-rank diagonal matrix $\Lambda$ such that $\Lambda Y = \ker(X)$, motivated our definition of self-projectivity. 
In particular, $X$ is self-associated if and only if $X$ is self-dual. The main difference between self-dual and self-associated points is the ambient space where they are considered: the first is considered in $X(k,2k)$, while the second is considered in $(\PP^{k-1})^{2k}$. We can move between these via the {\em Gelfand-MacPherson correspondence}~\cite{GM, kapranov93chow}
\begin{align*}
    \varphi: (\PP^{k-1})^n \dashrightarrow \Gr(k,n)/(\KK^*)^{n-1}, \quad (p_1,\ldots,p_n)\mapsto [\rowspan(p_1 \, \ldots \, p_n)]
\end{align*}
where $(p_1 \, \ldots \, p_n)$ is the $k\times n$ matrix with columns $p_i$. The map is defined whenever the $n$ points do not lie on a vector space of dimension smaller than~$k$. The map is surjective, but it is injective only up to the diagonal action of $\PGL(k)$ on $(\PP^{k-1})^n$. That is, there is a birational map
\[\overline{\varphi}:\; \raisebox{-0.2cm}{$\PGL(k)$}\backslash(\PP^{k-1})^n \dashrightarrow \Gr(k,n)/\raisebox{-0.2cm}{$(\KK^*)^{n-1}$}. \]
By Gelfand and MacPherson, there is a natural diffeomorphism \cite[Proposition 2.2.2]{GM}, between the generic configurations in $(\PP^{k-1})^n$ and $X(k,n)$. 
This descends to a diffeomorphism between the self-projecting configuration space $X(k,n)^{\sd}$ and the generic configurations of points in the parameter space $X_{k-1,2^m,n}$, where $m=\binom{k+1}{2}-n+1$, as defined in~\Cref{sec : point conf hypersurfaces}.

A priori, the quotient $\Gr(k,n)/(\KK^*)^{n-1}$ is well-defined only in the open part $\Gr(k,n)^\circ$.
There are many possible choices to define the quotient on the non-closed orbits.
The solution we use is Kapranov's Chow quotient~\cite{kapranov93chow}, where $\Gr(k,n)\sslash(\KK^*)^{n-1}$ can be thought of as the space of limits of closures of generic orbits. For a detailed definition, see~\cite{kapranov93chow}.
We can restrict the Chow quotient $\Gr(k,n)\sslash(\KK^*)^{n-1}$ to the self-projecting Grassmannian.
This defines a compactification of the self-projecting configuration space $X(k,n)^{\sd}$. Keel and Tevelev~\cite{KT} showed that for the self-dual case this is equal to the tropical compactification, as in \cite[Section 6.4]{MS}.
By \cite[Theorem 2.2.4]{kapranov93chow} the Gelfand--MacPherson correspondence extends to an isomorphism of the Chow quotients 
\[
    \overline{\varphi}:\;\raisebox{-0.2cm}{$\PGL(k)$}\lsslash(\PP^{k-1})^n \rightarrow \Gr(k,n)\sslash\raisebox{-0.2cm}{$(\KK^*)^{n-1}$}. 
\]

\subsection{Point configurations on hypersurfaces $X_{r,d^m,n}$ and beyond}\label{sec : point conf hypersurfaces}

Point configurations contained in certain varieties, like hypersurfaces, also appear often in the literature \cite{ARS25,caminata2023determinantal,chan2021moduli,traves2024ten}. The self-dual configuration space is strictly related to the moduli space $X_{r,d^m,n}\subset (\PP^r)^n$ describing configurations of $n$ points in $\PP^r$ lying on $m$ hypersurfaces of degree $d$. This space was introduced by Caminata, Moon and Schaffler in~\cite{caminata2023determinantal} motivated by the study of the geometry of data sets.

Setting $r = k-1$, $d = 2$, and $m = \binom{k+1}{2}-n +1$, by~\cite[Prop 2.4]{caminata2023determinantal}, $X_{k-1,2^m,n}$ is cut out by the $n$-minors of $\nu(X)$. Thus, it consists of projective points $V$ such that $\nu(V)$ has a non-trivial right kernel. Recall by~\Cref{prop : relation with minors} that $\SGr(k,n)$ is the generic irreducible component of the vanishing locus of the $n$-minors of $\nu(V)$. After quotienting by the torus action, we obtain the following.
\begin{proposition}
    Under the Gelfand-MacPherson correspondence, the self-projecting Grassmannian $\SGr(k,n)$ corresponds to the generic irreducible component of $X_{k-1,2^m,n}$, with $m = \binom{k+1}{2}-n+1$, where generic means that the hypersurfaces are not degenerate, irreducible and intersect in the expected way.
    
    In particular, we obtain an isomorphism between the quotients 
    \[
        \SGr\left(k,\tbinom{k+1}{2}\right)\sslash\raisebox{-0.2cm}{$(\KK^*)^{\tbinom{k+1}{2}-1}$} \to \raisebox{-0.2cm}{$\PGL(k)$}\lsslash X_{k-1,2^1,\tbinom{k+1}{2}}.
    \]
\end{proposition}
\begin{proof}  
    For a point $X\in\SGr(k,n)$, the $n$-minors of second multi-Veronese matrix $\nu(X)$ all vanish.
    Therefore, when restricting the Gelfand-MacPherson correspondence to the open self-projecting Grassmannian $\SGr(k,n)^\circ$, the image is contained in $\raisebox{-0.1cm}{$\PGL(k)$}\backslash X_{k-1,2^m,n}.$  As the Gelfand-MacPherson correspondence is an isomorphism and $\SGr(k,n)$ is irreducible, it follows that the image is an irreducible component of $\raisebox{-0.1cm}{$\PGL(k)$}\backslash X_{k-1,2^m,n},$  where the quadratic hypersurfaces are non-degenerate, irreducible, and intersect in the expected way.  This extends to the level of the Chow quotients, when we take the closure. 
    If $n = \binom{k+1}{2}$, we have $m = 1$ and $X_{k-1,2^m,n}$ is irreducible by \cite[Theorem A]{caminata2023determinantal}.
\end{proof}

From \cite[Theorem A]{caminata2023determinantal}, we know that $X_{r,d^1,n}$ is a geometrically irreducible Cohen–Macaulay normal variety in $(\PP^r)^n$. When this is isomorphic to the self-projecting configuration space, we obtain the same properties for $\SGr(k,n)$.

\begin{corollary}
    For $n = \binom{k+1}{2}$, we have that $\SGr(k,n)$ is reduced, irreducible, and Cohen-Macaulay in $\Gr(k,n)$.
\end{corollary}
\begin{proof} 
    This statement follows directly from \cite{caminata2023determinantal} and can be found in detail in \cite[Theorem 3.5]{ARS25}.
    In \cite{ARS25}, the variety $Z_2(k,n)$ is introduced, as the variety of points $\rowspan(M)\in \Gr(k,n)$ where the second multi-Veronese  matrix of $M$ is not of full rank. This coincides with $\SGr(k,n)$ whenever $\SGr(k,n)$ is of codimension 1, i.e. whenever $n = \binom{k+1}{2}$.  
    From~\cite[Theorem 3.5]{ARS25}, we know that $Z_2(k,n)$ is reduced, irreducible and Cohen-Macaulay.
\end{proof}

For higher codimension of $\SGr(k,n)$ in $\Gr(k,n)$ the situation is very intricate. As described in \cite[Section 2]{caminata2023determinantal}, the structure of $X_{r,d^m,n}$ for $m>1$ is very complicated. For example, it can have  multiple top-dimensional irreducible components.
Therefore, it is unclear whether properties like normal or Cohen-Macaulay still hold and, via Gelfand-MacPherson, could be applied to  $\SGr(k,n)$. As the techniques used in \cite{caminata2023determinantal,ARS25} cannot be generalized for $m>2$, one would need much more heavy machinery to investigate this question. 

\begin{remark}
    Special instances of the self-projecting Grassmannian appear hidden in the literature. One example is the ABCT variety $V(3,n)$ for $n=6$, which coincides with $\SGr(3,6)$.
    The ABCT variety $V(3,n)$ was first introduced by Arkani-Hamed, Bourjaily, Cachazo and Trnka \cite{arkani2011unification} in the context of the spinor-helicity formalism in physics. In
    \cite[Section 4]{ARS25} the authors study the homology class and degree of $V(3,6) = \SGr(3,6)$, while \cite[Section 5]{ARS25} discusses the difficulties in defining a positive geometry structure. 
\end{remark}

\section{Geometric interpretations}\label{sec : geom int}

In this section, we study in detail some self-projecting Grassmannians and give a geometric interpretation for them. 
We set $\KK$ to be algebraically closed and omit it from the notation.

\subsection{9 self-projecting points in $\PP^3$} \label{sec : X49}

The first non-trivial example of the self-projecting Grassmannian, which differs from the self-dual case, is $\SGr(4,9)$. An element of $X(4,9)$ is a generic configuration of $9$ points in $\PP^3$ up to projective equivalence.
Given $9$ points in general position in $\PP^3$ there is a unique quadric surface passing through the points. If the point configuration is self-projecting, there exists a $1$-dimensional family of quadric surfaces through the points. 
The intersection of two linearly independent quadric surfaces in $\PP^3$ defines a quartic curve. For general quadric surfaces, this curve is irreducible and of genus 1, i.e., it is an elliptic curve. Hence, a generic $V\in X(4,9)^{\sd}$ defines an elliptic curve with $9$ marked points. That is, we obtain a map
\[ 
    \varphi: X(4,9)^{\sd} \to \mM_{1,9}. 
\]
However, the moduli space $\mM_{1,9}$ of elliptic curves with $9$ marked points is a $9$-dimensional variety, while $X(4,9)^{\sd}$ is $10$-dimensional. Informally, the extra dimension of the self-projective configuration space corresponds to the choice of an embedding for the elliptic curve. Formally, we have the following.

\begin{theorem}\label{thm : X49}
    The self-projecting configuration space $X(4,9)^{\sd}$ and the moduli space $\mM_{1,10}$ are birationally equivalent.
\end{theorem}
\begin{proof}
    Consider the open subset $U$ of $X(4,9)^{\sd}$ of points $V\in X(4,9)^{\sd}$ on an irreducible curve $C_V$ of genus 1. 
    Different representations of $V\in U$ define isomorphic elliptic curves in $\PP^3$. Consider the hyperplane $H$ through $p_1,p_2,p_3$ in $\PP^3$ and denote by $p_{10}$ the fourth intersection point of $H$ with $C_V$. Since the $9$ points are in general position, no four of them lie on a hyperplane. Thus, $p_{10}\neq p_i$ for all $i\in [9]$. Therefore, we can define a map $\psi: U \to \mM_{1,10}$.

    The map $\psi$ is injective on $U$. For any two configurations $V,W\in X(4,9)^{\sd}$, where
    $V = (p_1\ldots p_9)$ and $W= (q_1\ldots q_9)$,  and $\psi(V)=\psi(W)$, then the curves $C_V$ and $C_W$ and the linear systems $|p_1+p_2+p_3+p_{10}|$ and $|q_1+q_2+q_3+q_{10}|$ are isomorphic. Hence, there exists a map $A\in \PGL(4)$ realizing this isomorphism. Since the curves are isomorphic as marked curves in $\mM_{1,10}$, this implies that $A$ maps each $p_i$ to $q_i$. It follows that $V$ and $W$ are equivalent up to $\PGL(4)$-action. In particular, they describe the same point in $X(4,9)^{\sd}$.

    Since $\mM_{1,10}$ is irreducible and has the same dimension as $X(4,9)^{\sd}$, it follows that $\psi$ is a birational map between the two varieties.
\end{proof}

\begin{corollary}
    There is a birational map between $X(3,9)$ and $X(4,9)^{\sd}$ obtained by lifting the elliptic curve from $\PP^2$ to $\PP^3$.
\end{corollary}
\begin{proof}
    By \Cref{thm : X49}, it is enough to show that $X(3,9)$ is also birational to $\mM_{1,10}$. Consider the open subset $U$ of $X(3,9)$ given by points $V$ such that the $10\times 9$ multi-Veronese matrix obtained by applying the third Veronese embedding to each column of $V$ has full rank equal to $9$. Then, given $V\in U$, the $9$ points $(p_1,p_2,\ldots,p_9)\in \PP^2$ corresponding to the columns of $V$ lie on a unique elliptic curve $C_V$. Define the $10$th point $p_{10}$ to be the third point of intersection of $C$ with the line through $p_1$ and $p_2$. Note that $p_{10}\neq p_i$ for all $i$, since we are assuming that the $9$ points are in general position. Then we have a well-defined map $\psi:U\to \mM_{1,10}$ sending a point $V$ to $(C_V,p_1,\ldots,p_{10})$ as described above.

    The map $\psi$ is injective. If $V=(p_1\ldots p_9),W=(q_1\ldots q_9)\in X(3,9)$ are mapped to the same point in $\mM_{1,10}$, then, similarly to the proof of \Cref{thm : X49}, $C_V$ and $C_W$ and the linear systems $|p_1+p_2+p_{10}|$ and $|q_1+q_2+q_{10}|$ are isomorphic. Hence, there exists a map $B\in \PGL(3)$ realizing that isomorphism, sending $V$ to $W$. Thus, $V$ and $W$ represent the same point in $X(3,9)$. As both varieties are irreducible and of the same dimension, the statement follows.
\end{proof}

\begin{remark}
    It would be interesting to explore whether and how the birational map between $\mM_{1,10}$ and $X(4,9)^{\sd}$ translates to a correspondence between the respective tropical varieties. Tropical moduli spaces of curves have a rich combinatorial structure relating to graphs; see~\cite{chan2021moduli}. We expect the tropicalization of $X(4,9)^{\sd}$ to be governed by matroid subdivisions of self-projecting valuated matroids, see~\Cref{def : self-proj matroid}. Note that an explicit computation of the tropicalization of $X(4,9)^{\sd}$ is currently out of reach. Similarly, it would be interesting to study whether the birational map $\mM_{5,13} \dashrightarrow X(5,13)^{\sd}$ as in~\Cref{thm : X513} can be tropicalized.
\end{remark}

\subsection{10 self-projecting points in $\PP^3$} \label{sec : X410}

The self-projecting Grassmannian $\SGr(4,10)$ is a codimension $1$ variety inside $\Gr(4,10)$ given by $V\in \Gr(4,10)$ such that $\det(\nu(V))$ vanishes. 
In other words, $\SGr(4,10)$ corresponds to the variety of $10$ points in $\PP^3$ lying on a quadric surface. The following can be checked computationally, but it was proved first by White~\cite{white1988implementation}.

\begin{proposition}\label{prop : Sgr410 def}
    The self-projecting Grassmannian    $\SGr(4,10)\subset\PP^{\binom{10}{4}-1}$ is defined in Pl\"ucker coordinates by a degree $5$ polynomial with $138$ terms.
\end{proposition}

The problem of giving a synthetic construction determining when $10$ points in $\PP^3$ lie on a quadric surface is the content of the ``Bruxelles' Problem'' posed in 1825 and solved by Will Traves in~\cite{traves2024ten}.

\subsection{13 self-projecting points in $\PP^4$} \label{sec : X513}

If a point $V\in \Gr(5,13)$ lies in the self-projecting Grassmannian, then the $15\times 13$ matrix $\nu(V)$ has rank strictly smaller than $13$. Hence, the columns of $V$ define points lying in the intersection of three distinct quadric hypersurfaces. Since a generic genus $5$ curve can be written as the intersection of three quadric hypersurfaces in $\PP^4$ via its canonical model, we obtain a connection to the moduli space $\mM_{5,13}$.

\begin{theorem} \label{thm : X513}
    The configuration space $X(5,13)^{\sd}$ is birationally equivalent to the moduli space $\mM_{5,13}$ of genus $5$ curves with $13$ distinct marked points.
\end{theorem}
\begin{proof}
    Consider the open subset $U$ of $X(5,13)^{\sd}$ of points $V$ such that $\nu(V)$ has rank equal to $12$. Then, the columns of $V$ lie on a unique genus $5$ curve, given by the complete intersection of $3$ linearly independent quadrics in the left kernel of $\nu(V)$. Hence, there is a well-defined map $\psi: U\to \mM_{5,13}$. The map is injective because the columns of $V$ correspond to the canonical embedding of the marked points on the genus $5$ curve, whenever this is non-trigonal and non-hyperelliptic. Since $\mM_{5,13}$ is irreducible and the varieties have the same dimension, the claim follows.
\end{proof}

\section{Self-projecting matroids and their realization spaces} \label{sec : 4 matroids}

In this section, we study the combinatorial properties of self-projecting point configurations. We introduce self-projecting matroids and computationally study their realization spaces inside the self-projecting Grassmannian.

Recall from~\cite{GHSV24} that a self-dual point $V\in \Gr(k,2k)$ defines an {\em (identically) self-dual matroid}, i.e. a matroid equal to its dual. In this sense, self-dual matroids encode how Pl\"ucker coordinates can vanish without breaking Equation~\eqref{eq : selfdual}. Realization spaces of self-dual matroids inside the self-dual Grassmannian $\SGr(k,2k)$ where studied in~\cite[Section 3]{GHSV24}. In particular, there exist realizable self-dual matroids that cannot be realized by a self-dual point configuration.
Here we show how this extends beyond the case $n=2k$. 

\subsection{Self-projecting matroids}\label{sec : sp matroids}

We start by noticing that self-duality has many equivalent definitions. Recall that a matroid $\mM$ has the {\em disjoint basis property} if for every basis $B$ of $\mM$ there exists a basis $C$ of $\mM$ such that $B\cap C = \emptyset$. Moreover, we introduce the following.

\begin{definition}
    Given a matroid $\mM$ of rank $k$ on $[n]$, an element $e\in [n]$ is said to be a {\em half-coloop} if there exist two rank $k-1$ flats $F_1,F_2$ such that $F_1\cup F_2 = [n]\setminus \{e\}$.
\end{definition}

The choice of the name comes from the fact that $e\in [n]$ is a coloop of $[n]$ whenever $[n]\setminus \{e\}$ is a flat of rank $k-1$. 
\begin{lemma}
    For a rank $k$ matroid $\mM$ on $[2k]$, the following are equivalent:
    \begin{enumerate}[label=(\roman*)]
        \item $\mM$ is self-dual;
        \item $\mM$ satisfies the disjoint basis property;
        \item $\mM$ has no half-coloop.
    \end{enumerate}
\end{lemma}
\begin{proof}
    The equivalence of {\em (i)} and {\em (ii)} is trivial. It is easy to see that {\em (i)} implies {\em (iii)}. Conversely, let $\mM$ be a matroid with no half-coloop and let $B\in \binom{[2k]}{k}$ be a basis of $\mM$. Suppose by contradiction that $[2k]\setminus B$ is not a basis and let $F_2$ be any rank $k-1$ flat containing $[2k]\setminus B$. Then any $e\in B\setminus F_1$ is a half-coloop, contradicting the hypothesis. Therefore the claim follows.
\end{proof}

Whenever $n>2k$, property {\em (i)} is not defined, and the non-existence of a half-coloop, property {\em (iii)}, is stronger than the disjoint basis property {\em (ii)}. 
\begin{lemma}
    Let $\mM$ be a matroid with no half-coloop. Then $\mM$ has the disjoint basis property.
\end{lemma}
\begin{proof}
    Let $\mM$ be a matroid of rank $k$ on $n$ elements with no half-coloop and fix $B$ a basis of $\mM$. Suppose by contradiction that $\rank_\mM([n]\setminus B)<k$. Then there exists some flat $F_1$ of rank $k-1$ containing $[n]\setminus B$. Let $F_2$ be any $k-1$ subset of $B$ containing $F_1\cap B$. Then $F_1,F_2$ are two rank $k-1$ flats with $|F_1\cup F_2|=n-1$, against the assumption.
\end{proof}

A self-projecting point satisfies the property of not having a half-coloop.
\begin{lemma}\label{lem : almost generic}
    For $V\in \SGr^\bullet(k,n)$, the matroid $\mM_V$ has no half-coloop.
\end{lemma}
\begin{proof}
    Let $S,T\in \binom{[n]}{k-1}$. The entry-wise product of the rows of the cocircuit matrix $D_{k,n}(V)$ labeled by $S$ and $T$ give rise to the equation:
    \begin{equation} \sum_{i=1}^n \lambda_i \sign(S,i)\sign(T,i)q_{S\cup i}(V)q_{T\cup i}(V) = 0. \label{eq : defining equation plucker}\end{equation}
    Suppose by contradiction that $e\in [n]$ is a half-coloop of $\mM_V$. Then there exist independent $S,T\in \binom{[n]}{k-1}$, such that the union of the flats containing them is equal to $[n]\setminus \{e\}$. Hence, $e\in [n]$ is the only element $i$ such that $S\cup \{i\}$ and $T\cup \{i\}$ are bases of $\mM_V$. For this choice of $S,T$, the equation above becomes $\lambda_e q_{S\cup e}(V)q_{T\cup e}(V) = 0$, which implies $\lambda_e = 0.$ This is a contradiction.
\end{proof}

\begin{definition}\label{def : self-proj matroid}
    A matroid $\mM$ is {\em self-projecting} if it has no half-coloop.
\end{definition}

\begin{remark}
    Note that if $\mM$ is self-projecting, any matroid obtained from it by deleting loops and/or parallel elements is still self-projecting. Therefore, we only study simple self-projecting matroids. 
\end{remark}

For small values of $k$ and $n$, the number of isomorphism classes of self-projecting matroids is recorded in~\Cref{tab:matroids}, where the matroids are assumed to be simple whenever $k\neq 2$. The data on matroids in this table is obtained from  \texttt{polydb}~\cite{polydb:paper}, which makes the database of matroids developed by~\cite{BIMM12} available.

\begin{table}[h]
    \centering
    \begin{tabular}{c||c|c||c|c||c|c||c|c||c|c}
        $\raisebox{-0.2cm}{$n$}\backslash k$ & \multicolumn{2}{c||}{2} & \multicolumn{2}{c||}{3} & \multicolumn{2}{c||}{4}& \multicolumn{2}{c||}{5} \\
          & \kern-0.25em{\scriptsize$\begin{matrix}\text{mat.}\end{matrix}$}\kern-0.25em &\kern-0.25em {\scriptsize$\begin{matrix} \text{self-proj.}\\ \text{mat.}\end{matrix}$} \kern-0.25em & \kern-0.25em{\scriptsize$\begin{matrix}\text{mat.}\end{matrix}$}\kern-0.25em &\kern-0.25em {\scriptsize$\begin{matrix} \text{self-proj.}\\ \text{mat.}\end{matrix}$} \kern-0.25em &\kern-0.25em{\scriptsize$\begin{matrix}\text{mat.}\end{matrix}$}\kern-0.25em &\kern-0.25em {\scriptsize$\begin{matrix} \text{self-proj.}\\ \text{mat.}\end{matrix}$} \kern-0.25em &\kern-0.25em{\scriptsize$\begin{matrix}\text{mat.}\end{matrix}$}\kern-0.25em &\kern-0.25em {\scriptsize$\begin{matrix} \text{self-proj.}\\ \text{mat.}\end{matrix}$} \kern-0.25em \\
         \hline
         4 & 7 & 2 & & & & & &\\
         5 & 13 & 5 & & & & & &\\
         6 & 23 & 12 & 9 & 2 & & & &\\
         7 & 37 & 22 & 23 & 12 & & & &\\
         8 & 58 & 39 & 68 & 53 & 617& 13 & &\\
         9 & 87 & 63 & 383 & 363 & 185981 & 7365 & & \\
         10 & 128 & 99 & 5249 & 5224 & ? & ? & ? & 1042
    \end{tabular}
    \caption{Number of matroids and self-projecting matroids for small $(k,n)$ up to isomorphism.}
    \label{tab:matroids}
\end{table}

\begin{remark}
    The number of non self-projecting matroids for rank $2$ and $3$ seem to follow a pattern similar to the one given in \cite{integersequence}[A024206]. For matroids of rank $4$ on $9$ elements, the disjoint basis property is fairly common, as there are $128676$ matroids satisfying it.
\end{remark}

\subsection{Self-projecting matroid realization spaces} \label{sec : sp realizaitons}

As seen above, self-projecting point configurations give rise to self-projecting matroids. In the following we study the reverse question, that is, when and how are self-projecting matroids realized by self-projecting point configurations. 
We show computationally that most small realizable self-projecting matroids can be realized (and are generically realized) by a self-projecting point. This provides strong evidence that our definition of self-projecting matroid is the right generalization of self-duality. Note that we cannot expect all realizable self-projecting matroids to be realized by a self-projecting point configuration, as this fails even for self-dual matroids~\cite{GHSV24}. Formally, we study the following objects.

\begin{definition}
    A $k\times n$ matrix $X$ with entries in $\KK$ is a \emph{realization} of a matroid $\mM$ if the columns of $X$ indexed by $I$ form a basis of $\KK^k$ if and only if $I$ is a basis of $\mM$.
    The \emph{realization space} $\mR_\KK(\mM)$ of $\mM$ is the quotient of the matroid cell $\Gr(\mM)^\circ/(\KK^*)^n$ generated by all realizations of $\mM$ over $\KK$. The matroid $\mM$ is \emph{realizable over $\KK$} if $\mR_\KK(\mM)\neq\emptyset$.
    
    A realization $X$ of $\mM$ is a \emph{self-projecting realization} of $\mM$ if there exists $\lambda\in(\KK^*)^n$ such that $X\diag(\lambda)X^t = 0$.  The \emph{self-projecting realization space} $\mS_\KK(\mM)$ of $\mM$ is the quotient of the matroid cell $\SGr(\mM)^\circ/(\KK^*)^n$ generated by all self-projecting realizations of $\mM$ over $\KK$.
\end{definition}

Note the self-projecting realization space of a matroid might be strictly smaller than the intersection of its realization space with the self-projecting Grassmannian. In particular, whenever $\mM$ is a non self-projecting matroid of rank $k$ on $n$ elements, and $k$ and $n$ are such that $\SGr(k,n)=\Gr(k,n)$, then $\mS_{\KK}(\mM) =\emptyset$, while the latter will be equal to $\mR_\KK(\mM)$.

We study self-projecting realization spaces computationally for $\KK = \overline{\mathbb{Q}}$, which will be dropped from the notation. In the following we report on the results, and discuss the computational methods used in~\Cref{sec : computations}.

\subsubsection*{Self-projecting realization spaces in rank $2$}

Matroids of rank $2$ on $n$ elements can be described by a subset $L\subset [n]$ containing the loops and a partition $P_1,\ldots,P_r$ of $[n]\setminus L$ such that for every $i\in [r]$, $\rk(P_i)=1$ and $\rk(P_i\cup a)=2$ for every $a\in [n]\setminus (L\cup P_i)$. In particular, isomorphism classes of rank $2$ matroids are determined by the sizes of $L$ and $P_i$'s. We will assume $|P_1|\geq |P_2|\geq \cdots \geq |P_r|$.
With this notation it is easy to check that a rank $2$ matroid has a half-coloop if and only if $r\in \{2,3\}$ and $|P_r|=1$.

\begin{theorem}\label{thm : rank2}
    For any self-projecting matroid $\mM$ of rank $2$, the realization space and the self-projecting realization space agree, $\mS(\mM) = \mR(\mM)$.
\end{theorem}
\begin{proof}
    Note that we can assume $\mM$ is loopless, i.e. $L=\emptyset$, and that $P_i = \{k_{i-1}+1,\ldots,k_i\}$ for each $i\in [r]$, with $k_{i-1}=0$. First assume $r\geq 4$. Let $W$ be any realization of $\mM$ and consider the restriction $W'$ to columns labeled by $\{ k_{j-1}+1 \mid j\in [r]\}$. Since the corresponding matroid is the uniform matroid $\mU_{2,r}$, $W'$ will be generically a self-projecting point with respect to some $\lambda'\in (\KK^*)^r$. Let $i\in [r]$ and suppose that $w_j = a_{i,j}w_{k_{i-1}+1}$ for all $j\in P_i$. Then $W$ is self-projecting with respect to $\lambda\in (\KK^*)^n$ with $\lambda_j = \frac{1}{a_{i,j}|P_i|}\lambda'_i $ for $j\in P_i$.
    
    If $r=2$, then $\mM$ can be restricted to the matroid $\mM_1$ on four elements with two classes of two parallel elements each. If $r=3$, then $\mM$ can be restricted to the matroid $\mM_2$ on six elements with three classes of two parallel elements each. One can show that $\mS(\mM_i) = \mR(\mM_i)$ for $i=1,2$. As above, realizations of $\mM$ are generically self-projecting. 
\end{proof}

\subsubsection*{Self-projecting realization spaces in rank $3$}

For rank 3 matroids, the first matroid appears for which $\mS(\mM)\subsetneq \mR(\mM)$. This is the uniform matroid $\mM=\mU_{3,6}$. Recall that from now on all matroids are simple.

\begin{theorem}\label{thm : realization space rank 3}
    For all self-projecting rank $3$ matroids $\mM$ on up to $8$ elements for which our computations terminated, we have $\mS(\mM) = \mR(\mM)$, whenever $\mM$ is not the uniform matroid $\mU_{3,6}$.
\end{theorem}

The dimensions of the (self-projecting) realization spaces for matroids of rank $3$ are displayed in~\Cref{tab:realization_rank3}. The computations of the self-projecting realization spaces for $4$~matroids of rank $3$ on $8$ elements did not terminate. In these cases, we computed the realization spaces, which were twice of dimension~$4$, and twice of dimension $5$.
\begin{table}[h]
    \centering
    \begin{tabular}{c||c|c|c|c|c|c|c|c|c|c}
  $(n,\cdot)\setminus \dim$   &\, -1 \, & \, 0 \, & \, 1 \, & \, 2 \, & \, 3 \, & \, 4 \, & \, 5 \, & \, 6 \, & \, 7 \, & \, 8 \, \\ \hline
    $(6,\mR)$   &0 & 0& 0& 1& 1& 0 & 0 & 0 & 0 &0 \\
    $(6,\mS)$   &0 & 0& 0& 2& 0& 0 & 0 & 0 & 0 &0 \\
    \hline
    $(7,\mR)$   &1 & 1& 1& 3& 3& 1 & 1 & 1 & 0 &0 \\
    $(7,\mS)$   &1 & 1& 1& 3& 3& 1 & 1 & 1 & 0 &0 \\
    \hline
    $(8,\mR)$  & 2 & 2  & 5 & 11 & 12 & 11 & 5 & 3 & 1 & 1\\
    $(8,\mS)$   & 2 & 2 & 5 & 11 & 12 & 9 & 3 & 3 & 1 & 1\\
    \end{tabular}
    \caption{Dimensions of the (self-projecting) realization spaces of rank $3$ matroids on $[n]$ over $\overline{\mathbb{Q}}$.
    }
    \label{tab:realization_rank3}
\end{table}

From \cite[Theorem 1.1]{CL25} we know that the realization spaces of matroids of rank $3$ for $n\leq 11$ are smooth. Whenever the realization space $\mR(\mM)$ and the self-projecting realization space $\mS(\mM)$ agree, it follows that the self-projecting realization space is also smooth. 

\subsubsection*{Self-projecting realization spaces in rank $4$}

The results of the computations for matroids of rank 4 on 8 elements are content of \cite{GHSV24}, therefore we do not repeat them here.
The most interesting case that we can compute is that of rank~$4$ matroids on $9$ elements, as this is the first example of the self-projecting Grassmannian, that is not a self-dual Grassmannian as in \cite{GHSV24}. 

In this case the computations terminated for $5758$ matroids out of the $7364$ self-projecting matroids, which are not the uniform matroid. 

\begin{theorem}\label{thm : realization space rank 4}
    There are $7365$ isomorphism classes of self-projecting simple matroids of rank $4$ on $9$ elements. Of these, $7181$ are realizable over characteristic zero, $174$ matroids satisfy $\mS(\mM)=\mR(\mM)\neq \emptyset$, and at least $5400$ matroids we have $\mS(\mM)\subsetneq\mR(\mM)$.
    There are at least $2844$ realizable self-projecting matroids for which $\emptyset \neq \mS(\mM)\subsetneq\mR(\mM)$.
\end{theorem}
\textbf{Note:} For $1606$ of the self-projecting matroids we could not determine the self-projecting realization spaces because the computation did not terminate. 
{\begin{table}[h]
    \centering
    {\small\begin{tabular}{c|cccccccccccccc}
         & -1 & 0 & 1 & 2 & 3 & 4 & 5 & 6 & 7 & 8 &  9 & 10 & 11 & 12 \\\hline
        {\scriptsize$\dim(\mR)$} &\kern-0.5em 184  \kern-0.5em&\kern-0.5em 19  \kern-0.5em&\kern-0.5em 194  \kern-0.5em&\kern-0.5em 850   \kern-0.5em&\kern-0.5em 1984 \kern-0.5em&\kern-0.5em  2175 \kern-0.5em&\kern-0.5em 1302 \kern-0.5em&\kern-0.5em  478  \kern-0.5em&\kern-0.5em 130 \kern-0.5em&\kern-0.5em  34  \kern-0.5em&\kern-0.5em 10  \kern-0.5em&\kern-0.5em  3  \kern-0.5em&\kern-0.5em  1   \kern-0.5em&\kern-0.5em 1\\
        {\scriptsize$\dim(\mS)$}  &\kern-0.5em 2740  \kern-0.5em&\kern-0.5em 71 \kern-0.5em&\kern-0.5em  758  \kern-0.5em&\kern-0.5em 1534 \kern-0.5em&\kern-0.5em  515  \kern-0.5em&\kern-0.5em 111 \kern-0.5em&\kern-0.5em  21 \kern-0.5em&\kern-0.5em   5  \kern-0.5em&\kern-0.5em  2  \kern-0.5em&\kern-0.5em  1  \kern-0.5em&\kern-0.5em  0  \kern-0.5em&\kern-0.5em  1  \kern-0.5em&\kern-0.5em  0   \kern-0.5em&\kern-0.5em 0\\
    \end{tabular}}
    \caption{Dimensions of the (self-projecting) realization spaces of rank $4$ matroids on $[9]$ over $\overline{\mathbb{Q}}$.}
    \label{tab:realizations49}
\end{table}}

It was already observed in \cite{GHSV24} that not every realizable self-dual matroid can be realized by a self-dual point configuration. This appeared for the first time for rank $5$ matroids on $10$ elements. When considering self-projecting matroids of rank $k$ on $n$ elements, we obtain a similar behavior. For the first time, within the scope of our computations, this appears for matroids of rank~$4$ on $9$ elements. There are at least $2556$ realizable self-projecting matroids for which $\mS = \emptyset.$ The distribution of the dimensions of the realization spaces of these matroids is displayed in~\Cref{tab:dim_nosp49}. We consider one example in detail.
\begin{table}[H]
    \centering
    \begin{tabular}{c|ccccccc}
         $\dim(\mR)$ &  0 & 1 & 2 & 3 & 4 & 5 & 6   \\\hline
        \# matroids &4 &  103 &  494  & 1089  & 738  & 124  & 4  \\
    \end{tabular}
    \caption{Dimensions of realization spaces $\mR(\mM)$ of self-projecting matroids of rank $4$ on~$[9]$ with~$\mS = \emptyset$.}
    \label{tab:dim_nosp49}
\end{table}

\begin{example}\label{ex : not_sp_realizable}
    Here, we give an example of a realizable, self-projecting matroid, which cannot be realized by a self-projecting point configuration. The matroid is given by the linear dependencies of the columns of the matrix
    \begin{equation*}
     M= \begin{pmatrix}
    1 &  0  & 0 &  0  & \frac{2}{3} &  0   &   1  & 1  & \frac{1}{2}\\
    0 &  1 &  0 &  0  &    0  & 2 &  \frac{1}{2} &  1 &  \frac{1}{2}\\
    0 &  0 &  1 &  0  &    1  & 1 &     1 &  1 &     1\\
    0  & 0 &  0 &  1  &    2  & 2 &     2 &  1 &     1\\
    \end{pmatrix}
    \end{equation*} 
    This matroid has zero-dimensional realization space. The only realization over characteristic $0$, up to $\PGL(4)$, is given by the matrix above. The self-projecting realization space is empty, because the second multi-Veronese matrix $\nu(M)$ has full rank. 
    This example can be found in the database in the file \verb|r_4_n_9_index_5985.mrdi|.
\end{example}

\subsection{The computations, the code and the database}\label{sec : computations}

The proofs of the above theorems are computational. The usage of computational methods and their resulting data in mathematics has been developing quickly within the last decades. In particular, the emphasis on making these data FAIR (Findable, Accessible, Interoperable, Reusable) and on making computations reproducible and transparent increased \cite{MardiWhitePaper}. We believe that our approach addresses the paramount questions and challenges on FAIR data and confirmable workflows.

We first describe the methods and the set-up used in the computations, thus providing the proof for the theorems. Then we describe how the code and output of these computations is made available to the research community, and how future research can profit from the created infrastructure.

\medskip

\begin{proof}[Proof and Discussion of \Cref{thm : realization space rank 3} and \Cref{thm : realization space rank 4}]
We extracted all self-projecting matroids from all rank $k$ matroids on $n$ elements stored in the database \texttt{polyDB}~\cite{polydb:paper} by computationally checking each matroid for the existence of half-coloops.
For the computation of the (self-projecting) realization spaces, we used the methods described in the proofs of \cite[Proposition 3.3 and Theorem~3.7]{GHSV24}, and adapted the code available on GitHub\footnote{\url{https://github.com/sachihashimoto/self-dual}} to our purposes.

The main idea is to use Gröbner basis computations in \texttt{Magma}. For each matroid in the list, we chose an isomorphic matroid for which $[k]$ is a basis. Note that in this setting the matroid is realizable if and only if it can  be realized by a $k \times n$ matrix of the form $X = ( \Id_k | x_{ij} )$, where $\Id_k$ is the identity matrix.
Since we consider the realizations over characteristic zero, the underlying polynomial ring is over $\mathbb{Q}.$
\Cref{algo:R} and \Cref{algo:S} describe the procedure. 

\begin{algorithm}[h] \label{alg:one}
	\caption{Computing the realization space $\mR(\mM)$}
	\label{algo:R}
	\begin{algorithmic}[1]
		\Require{A self-projecting matroid $\mM$ with $[k]$ one of the bases}
		\Ensure{Gr\"obner bases of $\mR(\mM)$, polynomials on which to localize}
        \State $X = ( \Id_k | x_{ij} )$, where
        $(x_{ij})_{i \in [k], j \in [n-k]}\in\text{Mat}(k, n-k)$
        \State $I' = \langle \det((x_{ij})_{i\in[k], j\in I}) | I \text{ nonbasis of }\mM  \rangle \subset \mathbb{Q}[x_{ij} : i \in [k], j \in [n-k]]$
       \Comment{ \textcolor{black!70}{\textit{This sets some of the entries $x_{ij}$ to zero} }}
        \State  Consider the remaining nonzero entries $x_{kl}$ in the $k \times (n-k)$ matrix of unknowns, and put $x_{kl} - 1$ into the ideal $I'$ for the first such nonzero entry in each column and row. 
        \State  $I_{\mR(\mM)} = (I': \prod_{I \text{ basis of }\mM}  \det(X_{ij})_{i\in[k], j\in I} ) ^\infty$ 
        \State \Return Gr\"obner basis of $I_{\mR(\mM)}$ in the localization of the polynomial ring at $\{\det(x_{ij})_{i\in[k], j\in I} |  I \text{ basis of }\mM\}$
	\end{algorithmic}
\end{algorithm}

\begin{algorithm}[h] \label{alg:two}
	\caption{Computing the self-projecting realization space $\mS(\mM)$}
	\label{algo:S}
	\begin{algorithmic}[1]
		\Require{A self-projecting matroid $\mM$ with $[k]$ one of the bases}
		\Ensure{Gr\"obner basis of $\mS(\mM)$, polynomials on which to localize}
        \State $X = ( \Id_k | x_{ij} )$, where 
        $(x_{ij})_{i \in [k], j \in [n-k]}\in\text{Mat}(k, n-k)$
        \State Compute $I_{\mR(\mM)} $ using \Cref{algo:R}.
        \State $J = I_{\mR(\mM)} +  \langle (X\cdot \Lambda \cdot X^t)_{i,j} | i,j \in [k] \rangle$, where $\Lambda = \diag(\lambda_1,\ldots,\lambda_n)$
        \State $J^{sat} = (I: \langle\prod_{i \in [n]}  \lambda_i\rangle ) ^\infty$ 
        \State $J_{\mS(\mM)} =J^{sat}\cap \mathbb{Q}[x_{ij} | i \in [k], j \in [n]]$
        \State $I_{\mS(\mM)} = (J_{\mS(\mM)}: \prod_{I \text{ basis of }\mM}  \det(x_{ij})_{i\in[k], j\in I} ) ^\infty$ 
        \State \Return Gr\"obner basis of $I_{\mS(\mM)}$ in the localization of the polynomial ring at $\{\det(x_{ij})_{i\in[k], j\in I} |  I \text{ basis of }\mM\}$
	\end{algorithmic}
\end{algorithm}

The algorithms were implemented in \texttt{Magma} \texttt{2.27}, parallelized with \texttt{GNU Parallel} \cite{gnu} and executed on a server with 2x 8-Core Intel Xeon Gold 6144 at 3.5 GHz (max: 4.2 GHz) and 768 GB RAM with NVidia RTX 4000 SFF Ada 20GB.

This implementation was successful, and near to instantaneous, for all matroids of rank $2$ on up to $12$ elements, for all rank $3$ matroids on up to $7$ elements and for the rank $4$ matroids on $8$ elements. For rank $3$ on $8$ elements and rank $4$ on $9$ elements the code was slower and for many matroids the computation did not terminate within the given timeout of $360$ seconds. The reason is the growing complexity of the Gr\"obner basis computation and saturation with respect to the $\lambda_i$. To be able to still get results, we use the same optimization trick as employed in \cite{GHSV24} for the computation of the self-dual realization spaces of the self-dual matroids of rank 5: the choice of an isomorphic matroid for $\mM$ with a \emph{frame}, i.e., a circuit $c$ of size $k+1$ containing the basis $[k]$. This is possible for all self-projecting matroids of rank $3$ on $8$ elements and all but two self-projecting matroids of rank $4$ on $9$ elements.
The isomorphism is a relabeling of the ground set of $\mM$ and this impacts the runtime of the calculations. 

Using this optimized code, we were able to compute the self-projecting realization spaces for all but 5 matroids of rank $3$ on $8$ elements with a timeout of $360$ seconds, and one more terminated within a couple of hours, while the last $4$ matroids did not terminate. For the self-projecting matroids of rank $4$ on $9$ elements  $5479$ terminated within $360$ seconds, while another $277$ terminated within $7000$ seconds. In the other cases, we computed $\mR$, which in all cases was nearly instantaneous.
\end{proof}

The code and output files from the $\texttt{Magma}$ computations are available on GitHub\footnote{\label{github}
\url{https://github.com/AlheydisGeiger/selfprojectingGrassmannian}}.
Since $\texttt{Magma}$ is not an open source software, the (self-projecting) realization spaces data will be available on \texttt{OSCAR} \cite{OSCAR,OSCAR-book}, in the newly emerging database collection \texttt{oscarDB}\footnote{\url{https://docs.oscar-system.org/dev/Experimental/OscarDB/introduction/}}.

The advantages of incorporating everything into the infrastructure of \texttt{OSCAR} are manifold: 
the database becomes findable and accessible from any {\tt OSCAR} installation; the large supporting community ensures that its maintenance and the associated code are as reliable as an open-source solution allows; development was aligned with existing implementations for matroid realization spaces within {\tt OSCAR} to maximize user-friendliness; and, finally, future reuse of our results, e.g., for checking irreducibility or smoothness of realization spaces, as well as code improvements as new methods emerge, are straightforward.
For example, from the work for \cite{CL25} there exists code to check smoothness, irreducibility and connectedness of matroid realization spaces using \texttt{OSCAR}. Since  the actual realization spaces were not stored, our code and database collection provide options to confirm their results and open perspectives for future combinations of the functionalities.

The \texttt{OSCAR} code to work with self-projecting realization spaces and to load the database entries is currently available on GitHub\footnote{
\href{https://github.com/AlheydisGeiger/Oscar.jl/tree/ag/selfprojecting_matroids/experimental/MatroidRealizationSpaces/src}{{\tt https://github.com/AlheydisGeiger/Oscar.jl/tree/ag/selfprojecting\_matroids}}} and will in the near future be incorporated into the experimental section of the main \texttt{OSCAR} distribution.  We also wrote some functions to reproduce the results from \texttt{Magma} in \texttt{OSCAR}. Code snippets and examples how to use the code, can be found in the project GitHub repository.

\section{Real and Positive} \label{sec : reality and positivity}

Real and positive versions of Grassmannians and their subvarieties lie at the heart of applications in the sciences. In this section we formulate some properties and questions regarding the real and positive self-projecting Grassmannian. We fix $\KK = \RR$ and omit it from the notation.

\subsection{Real and positive self-projecting Grassmannian} \label{sec : positive sp}

In the recent years, totally non-negative orthogonal Grassmannians appear in applications in particle physics and cosmology~\cite{Galashin_2020,maazouz2024positive,MR2607588}. However, as noted in~\cite[Section 5]{maazouz2024positive}, the structure and topological properties of these semialgebraic sets are not necessarily well-behaved whenever $n>2k$. Even in the $n=2k$ case, the totally positive part of the self-dual Grassmannian does not inherit the good behaviour of the orthogonal case, as noted for $\SGr(3,6)$  in the context of ABCT varieties~\cite[Section 5]{ARS25}. 
Note that since $\SGr(k,n)$ is not a homogeneous variety, we cannot consider its positive part in the sense of Lusztig. Therefore, we do not expect the totally non-negative self-projecting Grassmannian to behave nicely.

Recall that the real self-projecting Grassmannian is given by
\begin{equation} \label{eq : real selfprof}
    \SGr^\bullet_\RR(k,n) = \bigcup_{\lambda \in (\RR^*)^n} \OGr_\RR^\lambda(k,n). 
\end{equation}
The behavior of $\OGr_\RR^\lambda(k,n)$ depends on the choice of $\lambda$. We write $\Lambda_\RR(k,n)$ for the set of $\lambda\in\{\pm1\}^n$ for which $\OGr^\lambda(k,n)\neq \emptyset$. From Witt's theory, we have the following the following characterization of $\Lambda_\RR(k,n)$. 
\begin{proposition}
    Let $\lambda \in (\RR^*)^n$. The real orthogonal Grassmannian $\OGr^\lambda(k,n)$ is non-empty if and only if $s(\lambda)\geq k$, where $s(\lambda)$ is the signature of $\lambda$, i.e. the minimum between the number of positive and negative entries in $\lambda$.
\end{proposition}

Recall that the totally non-negative (resp. totally positive) Grassmannian $\Gr_{\geq 0}(k,n)$ (resp. $\Gr_{>0}(k,n)$) is the semialgebraic set defined as the intersection of $\Gr(k,n)$ in the dual Pl\"ucker embedding with the non-negative (resp. positive) orthant of $\PP^{\binom{n}{k}-1}$. Not all non-empty real orthogonal Grassmannians intersect the totally positive part. 

\begin{proposition}\label{prop : lambda positive}
    Suppose $V\in \Gr_{>0}(k,n)$  with $V\in \OGr^{\lambda}(k,n)$ for some $\lambda \in \Lambda_\RR(k,n)$. Then there is $S\in \binom{[n]}{2k}$, such that $\lambda|_S = \pm((-1)^i)_{i\in [2k]}$. 
\end{proposition}
\begin{proof}
    Let $V\in \Gr_{>0}(k,n)$ with $V\in \OGr^{\lambda}(k,n)$ for some $\lambda\in \Lambda_\RR(k,n)$. Let $i_1,\ldots,i_r$ be such that $\lambda$ changes sign at positions $i_1,\ldots,i_r$. Assume by contradiction that $r<2k-1$. If $r < 2k-2$ define $i_{r+1},\ldots,i_{2k-2}$ to be 
    equal to the last $2k-2-r$ elements of $[n]\setminus\{i_1,\ldots,i_r\}$. Let $S=\{i_1,\ldots,i_{k-1}\}$ and $T=\{i_{k},\ldots,i_{2k-2}\}$, we claim that in Equation \eqref{eq : defining equation plucker} corresponding to this choice of $S,T\in \binom{[n]}{k-1}$, all the non-zero coefficients have the same sign. Since the Pl\"ucker coordinates of $V$ are positive by assumption, this implies that Equation~\eqref{eq : defining equation plucker} does not vanish. This contradicts $V\in \OGr^{\lambda}(k,n)$.

    Let $i\in [n]\setminus (S\cup T)$. If $i<i_1$, then $\lambda_i\cdot \sign(S,i)\cdot\sign(T,i) = 1$. Suppose $i_1<i<i_2$, then both $\lambda_i$ and $\sign(S,i)$ will change sign. Therefore, the product remains equal to $1$. Proceeding this way it is possible to show that $\lambda_i\cdot \sign(S,i)\cdot\sign(T,i) = 1$ for all $i\in [n]\setminus (S\cup T)$.
\end{proof}

It is possible to restrict the self-projecting Grassmannian to the totally non-negative Grassmannian and obtain a notion of total non-negativity in this case. At the same time, by using Equation~\eqref{eq : real selfprof} one can consider the union of all $\OGr^\lambda_{\geq 0}(k,n)$ for all $\lambda\in \Lambda_\RR(k,n)$.
\begin{question}
    What is the best way to define the totally non-negative self-projecting Grassmannian and how do they relate? Can one show that
    \[
    \SGr_{\RR}(k,n)\cap \PP^{\binom{n}{k}-1}_{\geq 0} = \overline{\bigcup_{\lambda\in \Lambda_\RR(k,n)}\OGr^\lambda_{\geq 0}(k,n)} ?
    \]
    Given the current interest in the possibility of defining a positive orthogonal Grassmannian according to Lusztig, we leave it to future research to determine whether it is possible to construct a positive self-projecting Grassmannian from the union of these.
\end{question}

\subsection{Self-projecting positroids} \label{sec : sp positroids}

Recall that a positroid is a matroid that can be realized by a point in the totally non-negative Grassmannian $\Gr_{\geq 0}(k,n)$. Self-dual positroids where studied in~\cite{Galashin_2020} and correspond to fixed point-free involutions. A more general study of positroid cells inside the orthogonal Grassmannian was initiated in~\cite{maazouz2024positive}. In particular, we recall and generalize their notion of orthopositroids as follows.

\begin{definition}
    Given a positroid $\mM$ on $[n]$ of rank $k$, with bases $\mB$, and $\lambda\in \Lambda_\RR(k,n)$, consider for all $S,T \in \binom{[n]}{k-1}$ the sets
    \[ A^{\pm,\lambda}_{S,T} = \{ \ell \in [n] \mid S\cup \ell, T\cup \ell \in \mB  \text{ and } \lambda_\ell\sign(S,\ell)\sign(T,\ell) = \pm 1 \}. \]
    We say that $\mM$ is an {\em orthopositroid with respect to $\lambda$} if for every $S,T\in \binom{[n]}{k-1}$ we have $A^{+,\lambda}_{S,T} = \emptyset$ if and only if $A^{-,\lambda}_{S,T}= \emptyset$.
\end{definition}

Clearly, in order for any positroid to be realizable by an isotropic totally non-negative space it has to be an orthopositroid. Self-projecting positroids are in general a larger class than orthopositroids with respect to some $\lambda$. The number of isomorphism classes of positroids, self-projecting positroids and orthopositroids with respect to some $\lambda\in \Lambda_\RR(k,n)$ is recorded in~\Cref{tab:positroids}. The numbers were computed in {\tt sagemath} and the code is available at the GitHub repository\footref{github}.
Since every rank $2$ matroid is isomorphic to a positroid, we omit the numerical results for $k=2$, which are equal to the ones in~\Cref{tab:matroids}.
\begin{table}[h]
    \centering
    \begin{tabular}{c|| c|c|c || c|c|c || c|c|c ||}
        $\raisebox{-0.2cm}{$n$}\backslash k$ & \multicolumn{3}{c||}{3} & \multicolumn{3}{c||}{4}& \multicolumn{3}{c||}{5} \\
          & \kern-0.15em{\scriptsize$\begin{matrix}\text{pos.}\end{matrix}$}\kern-0.15em &\kern-0.3em {\scriptsize$\begin{matrix} \text{self-proj.}\\ \text{pos.}\end{matrix}$} \kern-0.3em &\kern-0.15em  {\scriptsize$\begin{matrix} \text{orth.}\end{matrix}$}& {\scriptsize$\begin{matrix}\text{pos.}\end{matrix}$} \kern-0.15em & \kern-0.3em {\scriptsize$\begin{matrix} \text{self-proj.}\\ \text{pos.}\end{matrix}$}\kern-0.3em & \kern-0.15em {\scriptsize$\begin{matrix} \text{orth.}\end{matrix}$}\kern-0.15em & {\scriptsize$\begin{matrix}\text{pos.}\end{matrix}$}\kern-0.15em &\kern-0.3em {\scriptsize$\begin{matrix} \text{self-proj.}\\ \text{pos.}\end{matrix}$} \kern-0.3em& \kern-0.15em {\scriptsize$\begin{matrix} \text{orth.}\end{matrix}$} \kern-0.15em\\
         \hline
         6 & 8 & 2 & 2 & & & & & & \\
         7 & 13 & 5 & 5 & & & & & & \\
         8 & 23 & 13 & 13 & 124 & 6 & 6 & & & \\
         9 & 38 & 26 & 26 & 408 & 30 & 29 & & & \\
         10 & 64 & 50 & 50 & 1301 & 200 & 200 & 5270 & 19 & 19
    \end{tabular}
    \caption{Number of simple positroids, self-projecting positroids and orthopositroids for small $(k,n)$ up to isomorphism.}
    \label{tab:positroids}
\end{table}

\begin{remark}
    For $(k,n) = (4,9)$, there exists one self-projecting positroid which is not an orthopositroid with respect to any choice of $\lambda\in \Lambda_\RR(4,9)$. The positroid $\mM$ has non-bases $\{1,2,3,4\},\{4,5,6,7\},\{1,7,8,9\}$. In particular this implies that there is no totally non-negative self-projecting realization of this positroid. That $\mM$ is not an orthopositroid follows from the fact that orthopositroids of type $(k,2k+1)$ have a $2k$-subset $S$ of the ground set such that the restriction $\mM_S$ is a self-dual positroid. The computation for the realization space of this positroid did not terminate, but it is possible to find self-projecting vector spaces representing the positroid, i.e. $\mS_\mM\neq \emptyset$.

    It is interesting to notice that for $k=4$ and $n=10$, we recover the equality between self-projecting positroids and orthopositroids.
\end{remark}

\bibliographystyle{alpha}
\bibliography{selfproj}

\end{document}